\documentclass{amsart}
\usepackage[utf8]{inputenc}
\usepackage{amssymb,amsmath,amsthm,mathrsfs,layout,graphicx}
\newtheorem{theorem}{Theorem}
\newtheorem{lemma}[theorem]{Lemma}

\newtheorem{proposition}[theorem]{Proposition}
\newtheorem{conjecture}[theorem]{Conjecture}
\newtheorem{example}[theorem]{Example}
\theoremstyle{definition}
\newtheorem{definition}[theorem]{Definition}
\newtheorem{remark}[theorem]{Remark}
\DeclareMathOperator{\Stab}{Stab}
\DeclareMathOperator{\Irr}{Irr}
\DeclareMathOperator{\Ind}{Ind}
\DeclareMathOperator{\Res}{Res}

\DeclareMathOperator{\GL}{GL}
\DeclareMathOperator{\tr}{tr}
\DeclareMathOperator{\Aut}{Aut}
\DeclareMathOperator{\charac}{char}
\newcommand{\oo}{\mathfrak{o}}

\title[Representations of Automorphism Groups of $\mathfrak{o}$-modules of type $(\ell,1^n)$]{The Representations of Automorphism Groups of $\mathfrak{o}$-modules of type $(\ell,1^n)$}
\author{Alexander Jackson}
\date{\today}

\begin{document}
\begin{abstract}
Let $\mathfrak{o}$ be the valuation ring of a non-Archimedean local field with finite residue field. We give a procedure to find the representation zeta polynomial of $\Aut_\mathfrak{o}(\mathfrak{o}_\ell\oplus\mathfrak{o}_1^{\oplus n})$ by induction on $n$. In particular, we show that the dimensions of the representations are given by evaluating finitely many polynomials at $q=|\mathfrak{o}_1|$.
\end{abstract}
\maketitle
\begin{figure}[b]
    \makebox[\textwidth][l]{
        \includegraphics[width=0.3\textwidth]{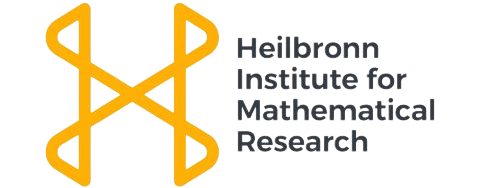}
        \includegraphics[width=0.15\textwidth]{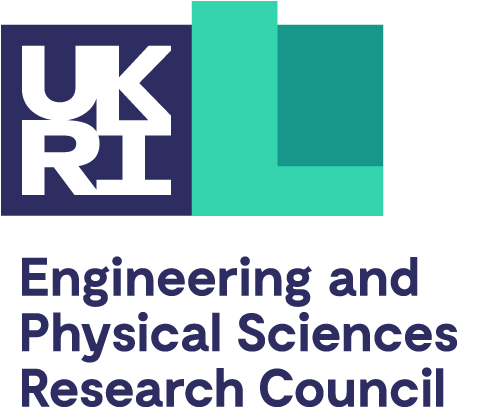}
    }
\end{figure}
\section{Introduction}
Let $\mathfrak{o}$ be the valuation ring of a non-Archimedean local field with maximal ideal~$\mathfrak{p}$ and finite residue field $\mathfrak{o}_1$. Let $n,r\geq 1$ be integers, and write $\mathfrak{o}_r=\mathfrak{o}/\mathfrak{p}^r$. Given a finite group $G$, denote by $\Irr(G)$ the set of irreducible complex representations of~$G$ up to isomorphism. The problem of constructing the representations of $\GL_n(\mathfrak{o}_r)$ has been studied in a series of papers by Hill \cite{hilljordan,hillnilpotent,hill95,hillsemisimplecuspidal}, in which a method is developed to construct certain classes of representations using Clifford theory relative to congruence subgroups of $\GL_n(\mathfrak{o}_r)$. For a detailed history of the problem, see \cite{stasinskisurvey}.
\par Define a \textbf{partition} of a natural number $n$ to be a decreasing sequence of natural numbers $(\lambda_1,\dots,\lambda_t)$ with $n=\sum\lambda_i$. Given any partition $\lambda=(\lambda_1,\dots,\lambda_t)$, Onn~\cite{onn08} defines $G_{\lambda,\mathfrak{o}}=\mathrm{Aut}_\mathfrak{o}(\mathfrak{o}_{\lambda_1}\oplus\dots\oplus\mathfrak{o}_{\lambda_t})$; note that when $\lambda_i=r$ for all $i$, then~$G_{\lambda,\mathfrak{o}}=\GL_n(\mathfrak{o}_r)$. Onn conjectures that the irreducible representations of $G_{\lambda,\mathfrak{o}}$ have dimensions which are polynomials in the residue cardinality $q$, and these occur with frequencies which are also polynomials in $q$. We suggest the following statement of Onn's conjecture, a version of which has already appeared in \cite{j24}:
\begin{conjecture}\label{Onn conjecture}
Let $\mathfrak{O}$ denote the set of rings which are the valuation ring of a non-Archimedean local field with finite residue field, up to isomorphism. Let $n\geq 1$ and a partition $\lambda$ of $n$ be given. There exist $k\geq 1$, polynomials \[d_1(x),\dots,d_k(x)\in\mathbb{Z}[x]\setminus\{0\},m_1(x),\dots,m_k(x)\in\mathbb{Q}[x]\setminus\{0\},\]
and for all $\mathfrak{o}\in\mathfrak{O}$ and $i\in\{1,\dots,k\}$, there exists $\mathcal{I}_{i,\mathfrak{o}}\subseteq\Irr(G_{\lambda,\mathfrak{o}})$
such that $\Irr(G_{\lambda,\mathfrak{o}})=\bigcup_{i=1}^k\mathcal{I}_{i,\mathfrak{o}}$ is a disjoint union,
and
\begin{enumerate}
\item $\text{for all }\rho\in\mathcal{I}_{i,\mathfrak{o}},\dim\rho=d_i(|\mathfrak{o}_1|)$, and
\item $|\mathcal{I}_{i,\mathfrak{o}}|=m_i(|\mathfrak{o}_1|)$.
\end{enumerate}
\end{conjecture}
This has been proven where $\lambda$ has length two by Onn \cite{onn08}, and for $\lambda=2^n$ by Singla~\cite{singla10}; $n\leq 4$. The statement for $\lambda=1^n$ is due to Green \cite{green}. In the present paper, we prove the statement of Conjecture \ref{Onn conjecture} where $\lambda=(\ell,1^n)$ and $\ell,n\geq 2$ are arbitrary.
\par In \cite{onn08}, Onn also conjectured:
\begin{conjecture}\label{Onn weak conjecture}
Let $\lambda$ be a partition and $\mathfrak{o},\mathfrak{o}'$ have equal residue cardinality $q$. Then
\[\mathbb{C}[G_{\lambda,\mathfrak{o}}]\cong\mathbb{C}[G_{\lambda,\mathfrak{o}'}],\]
or equivalently, a dimension-preserving bijection $\Irr(G_{\lambda,\oo})\leftrightarrow\Irr(G_{\lambda,\oo'})$.
\end{conjecture}
\par Since the group algebra of a finite group is determined by the dimensions of its irreducible representations, Conjecture \ref{Onn conjecture} implies Conjecture \ref{Onn weak conjecture}. In addition to the cases considered for Conjecture \ref{Onn conjecture}, Conjecture \ref{Onn weak conjecture} has been proven for $\lambda=2^n$ (for all $n$) by Singla \cite{singla10} and $\lambda=r^n$ (for all $r,n$ and $p=\charac\oo_1$ large) by Hadas \cite{hadas24}.
\par We describe the dimensions of the representations of the groups $G_{(\ell,1^n)}$. Define the \textbf{representation zeta polynomial} of a finite group $G$ to be the polynomial
\[\mathcal{R}_G(\mathcal{D})=\sum_{\rho\in\Irr(G)}\mathcal{D}^{\dim\rho}\in\mathbb{Z}[\mathcal{D}].\]
\begin{proposition}\label{representations of P_n}
The representation zeta polynomial of \[P_n=\begin{pmatrix}1&0&\dots&0\\ *&*&\dots&*\\\vdots&\vdots&\ddots&\vdots\\ *&*&\dots&*\end{pmatrix}\cong\mathfrak{o}_1^{n-1}\rtimes G_{(1^{n-1})}\] is given inductively by $\mathcal{R}_{P_1}(\mathcal{D})=\mathcal{D}$ and for $n\geq 2$,
\[\mathcal{R}_{P_n}(\mathcal{D})=\mathcal{R}_{P_{n-1}}(\mathcal{D}^{q^{n-1}-1})+\mathcal{R}_{\GL_{n-1}(\mathfrak{o}_1)}(\mathcal{D}).\]
\end{proposition}
\begin{proposition}\label{type 4 stabiliser}
The representation zeta polynomial of
\[T_n=\begin{pmatrix}1&*&*&\dots&*\\0&1&0&\dots&0\\0&*&*&\dots&*\\\vdots&\vdots&\vdots&\ddots&\vdots\\0&*&*&\dots&*\end{pmatrix}\cong\mathfrak{o}_1^{n-1}\rtimes(\mathfrak{o}_1^{n-2}\rtimes G_{(1^{n-2})})\]
is given inductively by $\mathcal{R}_{T_2}(\mathcal{D})=q\mathcal{D}$
and for $n\geq 3$,
\begin{align*}\mathcal{R}_{T_n}(\mathcal{D})&=(q-1)\mathcal{R}_{\GL_{n-2}(\mathfrak{o}_1)}(\mathcal{D}^{q^{n-2}})+\mathcal{R}_{T_{n-1}}(\mathcal{D}^{q^{n-2}-1})+\mathcal{R}_{P_{n-1}}(\mathcal{D}).\end{align*}
\end{proposition}
The proofs of Propositions \ref{representations of P_n} and \ref{type 4 stabiliser} use similar arguments as in the proof \linebreak of {\cite[Theorem 26.9]{huppert98}}, which shows that the irreducible representations of the unitriangular groups $U_n(\mathfrak{o}_1)$ have dimensions which are powers of $q$. 
\par Knowing $\mathcal{R}_{P_n}(\mathcal{D})$ and $\mathcal{R}_{T_n}(\mathcal{D})$, one can calculate $\mathcal{R}_{G_{(\ell,1^n)}}(\mathcal{D})$ as follows:
\begin{theorem}\label{main theorem}
The representation zeta polynomial of $G_{(\ell,1^n)}$ is given by
\begin{align*}\mathcal{R}_{G_{(\ell,1^n)}}(\mathcal{D})&=q^{\ell-2}(q-1)^2\mathcal{R}_{\GL_n(\mathfrak{o}_1)}(\mathcal{D}^{q^n})+q^{\ell-2}(q-1)\mathcal{R}_{((\mathfrak{o}_1^n\times\mathfrak{o}_1^n)\rtimes G_{(1^n)})}(\mathcal{D}),\end{align*}
where
\begin{align*}\mathcal{R}_{(\mathfrak{o}_1^n\times\mathfrak{o}_1^n)\rtimes G_{(1^n)}}(\mathcal{D})&=\mathcal{R}_{G_{(1^n)}}(\mathcal{D})+2\mathcal{R}_{P_n}(\mathcal{D}^{q^n-1})+\mathcal{R}_{T_n}(\mathcal{D}^{(q^n-1)(q^{n-1}-1)})\\&\quad+(q-1)\mathcal{R}_{G_{(1^{n-1})}}(\mathcal{D}^{q^{n-1}(q^n-1)}).\end{align*}
In particular, Conjecture \ref{Onn conjecture} holds for $\lambda=(\ell,1^n)$.
\end{theorem}
\par The representation zeta polynomial of $G_{(2,1,1)}$ was already known by Singla \cite{singla10}. In fact, if we can show that the dimensions and frequencies of representations of $C_{\GL_n(\mathfrak{o}_1)}(\beta)$ are given by polynomials in $q$ as in Conjecture \ref{Onn conjecture}, then the same is true for $\GL_n(\mathfrak{o}_2)$ for all $n$ (see \cite[\S 7.1]{singla10}).
\par The structure of the paper is as follows: in Section \ref{preliminaries}, we give a brief review of Clifford theory and the representation zeta polynomial. In Section \ref{application}, we explain how Theorem \ref{main theorem} relates to $\GL_n(\mathfrak{o}_r)$ and give proofs of Propositions \ref{representations of P_n} and \ref{type 4 stabiliser}. In Section \ref{arbitrary characteristic}, we derive the representation zeta polynomial of $G_{(\ell,1^n),\mathfrak{o}}$, thus proving Theorem \ref{main theorem}.
\section{Preliminaries and Clifford Theory}\label{preliminaries}
We outline the approach of Clifford theory introduced in \cite{clifford37}. Let $G$ be a finite group and $N$ a normal subgroup of $G$. Here, $G$ acts on $\Irr(N)$ as follows: given $g\in G,\psi\in\Irr(N)$, define $^g\psi\in\Irr(N)$ by
\[^g\psi(x)=\psi(g^{-1}xg)\]
for $x\in N$. We also define \[\Irr(G\mid\psi)=\{\rho\in\Irr(G)\mid\langle\psi,\rho|_N\rangle\neq 0\}.\]
\begin{definition}
Let $G$ be a finite group, $N$ a normal subgroup of $G$, and $\psi\in\Irr(N)$.
The \textbf{representation zeta polynomial of $G$ above $\psi$} is
\[\mathcal{R}_{G\mid\psi}(\mathcal{D})=\sum_{\rho\in\Irr(G\mid\psi)}\mathcal{D}^{\dim\rho}.\]
\end{definition}
For example, the representation zeta polynomial of $\GL_2(\mathfrak{o}_1)$ is:
\[\mathcal{R}_{\GL_2(\mathfrak{o}_1)}(\mathcal{D})=(q-1)\mathcal{D}+\frac 12 q(q-1)\mathcal{D}^{q-1}+(q-1)\mathcal{D}^q+\frac 12(q-1)(q-2)\mathcal{D}^{q+1}.\]
Defining $B$ to be the subgroup of upper triangular matrices in $\GL_2(\mathfrak{o}_1)$, there are two representations (the trivial representation of dimension one and the Steinberg representation of dimension $q$) whose restriction to $B$ contains the trivial representation of $B$:
\[\mathcal{R}_{\GL_2(\mathfrak{o}_1)\mid 1_B}(\mathcal{D})=\mathcal{D}+\mathcal{D}^q.\]
The following result is standard in Clifford theory; for proofs, see \cite[Theorems 6.2, 6.7, 11.22]{isaacs}:
\begin{proposition}\label{Clifford theory}
Let $G$ be a finite group and $N$ a normal subgroup of $G$.
For every $\rho\in\Irr(G)$, there exists a $G$-orbit $\Omega\subseteq\Irr(N)$ and an integer $e\geq 1$ such that
\[\Res^G_N\rho=e\bigoplus_{\psi\in\Omega}\psi.\]
Also, for every $\psi\in\Irr(N)$, there is a bijection
\[\Irr(\Stab_G(\psi)\mid\psi)\to\Irr(G\mid\psi)\]
given by induction, $\theta\mapsto\Ind^G_{\Stab_G(\psi)}\theta$.\\
If, further, $\psi$ extends to a representation $\hat{\psi}\in\Irr(\Stab_G(\psi))$, then there is a bijection
\[\Irr(\Stab_G(\psi)/N)\to\Irr(\Stab_G(\psi)\mid\psi)\]
given by $\overline{\theta}\mapsto\theta\otimes\hat{\psi}$, where $\theta$ is obtained from $\overline{\theta}$ by composition with the natural map~$\Stab_G(\psi)\to\Stab_G(\psi)/N$.
\end{proposition}
\begin{remark}\label{sum of zeta polynomials above representatives}
By Proposition \ref{Clifford theory}, if $X$ is any set of representatives of the $G$-orbits on $\Irr(N)$, then
\[\mathcal{R}_G(\mathcal{D})=\sum_{\psi\in X}\mathcal{R}_{G\mid\psi}(\mathcal{D}).\]
Also by Proposition \ref{Clifford theory}, \[\mathcal{R}_{G\mid\psi}(\mathcal{D})=\mathcal{R}_{\Stab_G(\psi)\mid\psi}(\mathcal{D}^{[G:\Stab_G(\psi)]}).\]
\par We shall also need to describe the representations of a semidirect product where the normal subgroup is abelian (this is Proposition 25 of \cite{serre71}):
\begin{proposition}\label{Clifford theory abelian normal}
Let $G=N\rtimes H$, where $N$ is abelian. Let $\chi_i$ be representatives of the orbits of $H$ on $\Irr(N)$. For $\rho\in\mathrm{Irr}(\mathrm{Stab}_H(\chi_i))$, define~${\tilde{\rho}\in\mathrm{Irr}(N\rtimes\mathrm{Stab}_H(\chi_i))}$ by composition with the projection \[N\rtimes\mathrm{Stab}_H(\chi_i)\twoheadrightarrow \mathrm{Stab}_H(\chi_i).\] Define $\tilde{\chi_i}\in\Irr(N\rtimes\mathrm{Stab}_H(\chi_i))$ by ${\tilde{\chi_i}(nh)=\chi_i(n)}$. The irreducible representations of $G$ are precisely
\[\mathrm{Ind}_{N\rtimes\mathrm{Stab}_H(\chi_i)}^G(\tilde{\chi_i}\otimes\tilde{\rho}),\]
and further, these are all distinct for distinct choices of pairs $(i,\rho)$.
\end{proposition}
Note that this is a special case of Proposition \ref{Clifford theory}, where the stabiliser also splits as a semidirect product, $\Stab_G(\chi_i)=N\rtimes\Stab_H(\chi_i)$. In order to find $\mathcal{R}_G(\mathcal{D})$, it is sufficient to find representatives of the $H$-orbits on $\Irr(N)$ and the representation zeta polynomial of each of the stabilisers of the orbit representatives.
\end{remark}
\section{Application to $\GL_n(\mathfrak{o}_r)$}\label{application}
We return to considering the representations of $\GL_n(\mathfrak{o}_r)$. We shall only consider the case $r=2$, which is the situation studied by Singla \cite{singla10}. Define the (abelian) congruence subgroup \[K^1=\{g\in\GL_n(\mathfrak{o}_2)\mid g\equiv 1\bmod \mathfrak{p}\}=1+\mathfrak{p} M_n(\mathfrak{o}_2).\]
The representations of $K^1$ are given as follows: fix some character ${\psi:\mathfrak{o}_2\to\mathbb{C}^\times}$ which is non-trivial on $\mathfrak{p}$. For $\beta\in M_n(\mathfrak{o}_1)$, define $\psi_\beta\in\Irr(K^1)$ by
\[\psi_\beta(1+x)=\psi(\tr(\hat{\beta} x)),\]
where $\hat{\beta}$ is a lift of $\beta$ to $M_n(\mathfrak{o}_2)$. Note in particular that $\psi_\beta$ does not depend on the choice of lift.
By Proposition \ref{Clifford theory}, each $\rho\in\Irr(\GL_n(\mathfrak{o}_2))$ defines an orbit on $\Irr(K^1)$, and since the action is given by $^g\psi_\beta=\psi_{\overline{g}\beta \overline{g}^{-1}}$, $\rho$ defines an orbit of the conjugation action of $\GL_n(\mathfrak{o}_2)$ on $M_n(\mathfrak{o}_1)$. By \cite[Proposition 2.2]{singla10}, for every $\beta$, $\psi_\beta$ extends to a representation $\hat{\psi}_\beta\in\Irr(\Stab_{\GL_n(\mathfrak{o}_2)}(\psi_\beta))$.
Thus, by Proposition \ref{Clifford theory}, the representations of the stabiliser containing $\psi_\beta$ are given by
\[\mathrm{Irr}(\mathrm{Stab}(\psi_\beta)\mid\psi_\beta)=\{\hat{\psi}_\beta\otimes\theta\mid\overline{\theta}\in\mathrm{Irr}(\Stab(\psi_\beta)/K^1)\}.\]
\par By \cite[Proposition 2.3(2)]{hill95}, we have $\Stab(\psi_\beta)/K^1\cong C_{\GL_n(\mathfrak{o}_1)}(\beta)$, so we look at representations of this centraliser. In fact,
\begin{proposition}
The representation zeta polynomial of $\GL_n(\oo_2)$ is determined by $\mathcal{R}_{C_{\GL_m(\oo_1)}(\beta)}(\mathcal{D})$ for $m\leq n$ and $\beta$ a nilpotent Jordan canonical form. If the $C_{\GL_m(\oo_1)}(\beta)$ have the polynomial property as in Conjecture \ref{Onn conjecture}, then the same is true for $\GL_n(\oo_2)$.
\end{proposition}
\begin{proof}
We outline an argument due to Hill \cite{hilljordan}. Let $\overline{s}\in M_n(\mathfrak{o}_1)$ be semisimple. For~${\beta\in M_n(\mathfrak{o}_1)}$, write $\beta\sim_\mathrm{s.s.} \overline{s}$ if the semisimple part of the Jordan-Chevalley decomposition of $\beta$ is conjugate to $\overline{s}$. The \textbf{geometric conjugacy class} of $\overline{s}$ is
\[\mathcal{C}_{\GL_n(\mathfrak{o}_r)}(\overline{s})=\{\rho\in\Irr(\GL_n(\mathfrak{o}_r))\mid\langle \psi_\beta,\rho|_{K^{r-1}}\rangle\neq 0\text{ for some }\beta\in M_n(\mathfrak{o}_1)\text{ s.t. }\beta\sim_\mathrm{s.s.}\overline{s}\}.\]
For each semisimple $\overline{s}\in M_n(\mathfrak{o}_1)$, there is a lift $s\in M_n(\mathfrak{o}_r)$ with the property that there exist unramified extensions $\mathfrak{o}^{(j)}$ of $\mathfrak{o}$ such that
\[C_{\GL_n(\mathfrak{o}_r)}(s)\cong\prod_{j=1}^t\GL_{m_j}(\mathfrak{o}^{(j)}/(\varpi^r)),\]
where $\sum_{j=1}^t m_j=n$ and $t=1$ if and only if $\overline{s}=aI$ for some $a\in\mathfrak{o}_1$.
By Theorem 2.13 of \cite{hilljordan}, there is a bijection
\[\mathcal{C}_{\GL_n(\mathfrak{o}_r)}(\overline{s})\to \mathcal{C}_{C_{\GL_n(\mathfrak{o}_r)}(s)}(0);\quad\rho\mapsto\rho_\mathrm{nil},\]
where we interpret the set $\mathcal{C}_{C_{\GL_n(\mathfrak{o}_r)}(s)}(0)$ to consist of tensor products of representations in each $\mathcal{C}_{\GL_{m_j}(\mathfrak{o}^{(j)}/(\varpi^r))}(0)$.
Further, for all $\rho\in\mathcal{C}_{\GL_n(\mathfrak{o}_r)}(\overline{s})$, $\dim\rho/\dim\rho_\mathrm{nil}$ is a polynomial in $q=|\oo_1|$ which depends only on the decomposition of $\oo_1^n$ into $\overline{s}$-invariant subspaces and is independent of $\rho$. This allows us to work inductively; if we know $\mathcal{C}_{\GL_m(\oo_r)}(0)$ for $m<n$, then we can construct $\mathcal{C}_{\GL_n(\oo_r)}(\overline{s})$, where $\overline{s}$ is not a scalar matrix. On the other hand, if $\rho\in\Irr(\GL_n(\oo_r)\mid\psi_{aI+N})$ where $N$ is nilpotent, then one can show that $\psi_{-aI}$ extends to $\tilde{\psi}_{-aI}\in\Irr(\GL_n(\oo_r))$, and that $\rho\otimes\tilde{\psi}_{-aI}\in\Irr(\GL_n(\oo_r)\mid\psi_N)$ with the same dimension as $\rho$.
\end{proof}
\par Furthermore, each of the centralisers $C_{\GL_n(\oo_r)}(\beta)$ is isomorphic to some $G_\lambda$:
\begin{proposition}\label{centraliser isomorphic to G lambda} \cite[Proposition 4.11]{singla10}
Let $\lambda$ be a partition of $n$ and $\beta_\lambda$ a nilpotent Jordan matrix of type $\lambda$. Then there is an isomorphism of groups
\[C_{\GL_n(\mathfrak{o}_1)}(\beta_\lambda)\cong G_{\lambda,\mathbb{F}_q[[t]]}.\]
\end{proposition}
Thus, if $G_{\lambda,\mathfrak{o}}$ has the polynomial property of Conjecture \ref{Onn conjecture} for all partitions $\lambda$ of $m$ with $m\leq n$, then the same is true for $\GL_n(\mathfrak{o}_2)$.

\section{Proofs of Propositions \ref{representations of P_n} and \ref{type 4 stabiliser}}

We fix the following notation throughout the subsequent proofs:
\[\mathbf{v}=\begin{pmatrix}v_1\\\vdots\\v_{n-2}\end{pmatrix},\mathbf{w}=\begin{pmatrix}w_1\\\vdots\\w_{n-1}\end{pmatrix},Y=\begin{pmatrix}y_{11}&\dots&y_{1,n-2}\\\vdots&\ddots&\vdots\\y_{n-2,1}&\dots&y_{n-2,n-2}\end{pmatrix}.\]
For $i\in\mathfrak{o}_1$, let $\sigma_i$ be the additive character of $\mathfrak{o}_1$ defined by $\sigma_i(x)=\zeta_p^{\tr(ix)}$, where~$\zeta_p$ is a fixed primitive $p$th root of unity in $\mathbb{C}$ and $\tr$ denotes the absolute trace $\tr_{\mathfrak{o}_1/\mathbb{F}_p}$. 
\begin{proof}[Proof of Proposition \ref{representations of P_n}]
Note $P_n=\tilde{N}_n\rtimes \tilde{H}_n$, where
\[\tilde{N}_n=\left\{\begin{pmatrix}1&0\\\mathbf{v}&I\end{pmatrix}\mid\mathbf{v}\in\mathfrak{o}_1^{n-1}\right\},\tilde{H}_n=\left\{\begin{pmatrix}1&0\\0&Y\end{pmatrix}\mid Y\in\GL_{n-1}(\mathfrak{o}_1)\right\}.\]
Parameterise the irreducible representations of $\tilde{N}_n$:
\begin{align*}\omega_{i_1,\dots,i_{n-1}}\begin{pmatrix}1&0\\\mathbf{v}&I\end{pmatrix}&=\sigma_{i_1}(v_1)\dots\sigma_{i_{n-1}}(v_{n-1})\\&=\zeta_p^{\mathrm{tr}(i_1v_1+\dots+i_{n-1}v_{n-1})}.\end{align*}
The conjugate of an element of $\tilde{N}_n$ by an element of $\tilde{H}_n$ takes the following form:
\[\begin{pmatrix}1&0\\&Y\end{pmatrix}\begin{pmatrix}1&0\\\mathbf{v}&I\end{pmatrix}\begin{pmatrix}1&0\\&Y\end{pmatrix}^{-1}=\begin{pmatrix}1&0\\Y\mathbf{v}&I\end{pmatrix}.\]
Writing $h=\begin{pmatrix}1&0\\0&Y\end{pmatrix}$,
\begin{align*}
\omega_{i_1,\dots,i_{n-1}}^h\begin{pmatrix}1&0\\\mathbf{v}&I\end{pmatrix}&=\omega_{i_1,\dots,i_{n-1}}\begin{pmatrix}1&0\\Y\mathbf{v}&I\end{pmatrix}\\
&=\zeta_p^{\mathrm{tr}(\theta(\mathbf{v}))},
\end{align*}
where
\begin{align*}
\theta(\mathbf{v})&=i_1(y_{11}v_1+\dots+y_{1,n-1}v_{n-1})\\
&\quad+\dots\\
&\quad+i_{n-1}(y_{n-1,1}v_1+\dots+y_{n-1,n-1}v_{n-1})\\
&=(i_1y_{11}+\dots+i_{n-1}y_{n-1,1})v_1\\
&\quad +\dots\\
&\quad+(i_1y_{1,n-1}+\dots+i_{n-1}y_{n-1,n-1})v_{n-1}.
\end{align*}
Therefore, the right action of $\tilde{H}_n$ on $\Irr(\tilde{N}_n)$ is given explicitly:
\[\omega_{i_1,\dots,i_{n-1}}^h=\omega_{(i_1y_{11}+\dots+i_{n-1}y_{n-1,1}),\dots,(i_1y_{1,n-1}+\dots+i_{n-1}y_{n-1,n-1})}.\]
The index transformation is given by the matrix $Y$, which is non-singular of free choice, and since every non-zero vector can be mapped to any other by multiplication by a non-singular matrix, there are two orbits, $\{\omega_{i_1,\dots,i_{n-1}}|i_k\neq 0\text{ for some }k\}$ and $\{\omega_{0,\dots,0}\}$. We find equations for the stabilisers:
\begin{align*}
\omega_{i_1,\dots,i_{n-1}}^h=\omega_{i_1,\dots,i_{n-1}}&\iff\text{for all } \mathbf{v}\in\mathfrak{o}_1^{n-1},\zeta_p^{\mathrm{tr}(\theta(\mathbf{v}))}=\zeta_p^{\mathrm{tr}(i_1v_1+\dots+i_{n-1}v_{n-1})}\\
&\iff\text{for all } \mathbf{v}\in\mathfrak{o}_1^{n-1},\zeta_p^{\mathrm{tr}(\theta'(\mathbf{v}))}=1,
\end{align*}
where
\begin{align*}
\theta'(\mathbf{v})&=\theta(\mathbf{v})-(i_1v_1+\dots+i_{n-1}v_{n-1})\\
&=(i_1(y_{11}-1)+\dots+i_{n-1}y_{n-1,1})v_1\\
&\quad+\dots\\
&\quad+(i_1y_{1,n-1}+\dots+i_{n-1}(y_{n-1,n-1}-1))v_{n-1}
\end{align*}
Setting all but one of the $v_j$ to zero in turn, we obtain the equations
\begin{align*}
i_1(y_{11}-1)+\dots+i_{n-1}y_{n-1,1}&=0\\
&\vdots\\
i_1y_{1,n-1}+\dots+i_{n-1}(y_{n-1,n-1}-1)&=0.
\end{align*}
Choose the representative $\omega_{1,0,\dots,0}$ of the non-trivial orbit:
\begin{align*}
\mathrm{Stab}_{\tilde{H}_n}(\omega_{1,0,\dots,0})&=\left\{\begin{pmatrix}1&&&\\&1&0&\dots&0\\&y_{21}&y_{22}&\dots&y_{2,n-1}&\\&\vdots&\vdots&\ddots&\vdots&\\&y_{n-1,1}&y_{n-1,2}&\dots&y_{n-1,n-1}\end{pmatrix}\middle|\begin{aligned}x,y_{jk}\in\mathfrak{o}_1,\\\det\neq 0\end{aligned}\right\}\\
&\cong P_{n-1},
\end{align*}
with index
\[\frac{|\GL_{n-1}(\mathfrak{o}_1)|}{|\mathfrak{o}_1^{n-2}\rtimes\GL_{n-2}(\mathfrak{o}_1)|}=q^{n-1}-1\]
in $\tilde{H}_n$. In the case of $n=2$, this is the trivial group. Applying Proposition \ref{Clifford theory abelian normal}, the representation zeta polynomial above $\omega_{1,0,\dots,0}$ is
\[\mathcal{R}_{P_n\mid\omega_{1,0,\dots,0}}(\mathcal{D})=\mathcal{R}_{P_{n-1}}(\mathcal{D}^{q^{n-1}-1}).\]
In the case of the trivial orbit, $\mathrm{Stab}_{\tilde{H}_n}(\omega_{0,\dots,0})=\tilde{H}_n\cong\mathrm{GL}_{n-1}(\mathfrak{o}_1)$, and
\[\mathcal{R}_{P_n\mid\omega_{0,\dots,0}}(\mathcal{D})=\mathcal{R}_{\GL_{n-1}(\mathfrak{o}_1)}(\mathcal{D}).\]
We can deduce the following formula for the representation zeta polynomial of $P_n$ ($n\geq 2$):
\[\mathcal{R}_{P_n}(\mathcal{D})=\mathcal{R}_{P_{n-1}}(\mathcal{D}^{q^{n-1}-1})+\mathcal{R}_{\GL_{n-1}(\mathfrak{o}_1)}(\mathcal{D}).\]
Note that $P_1=1$, therefore $\mathcal{R}_{P_1}(\mathcal{D})=\mathcal{D}$.
\end{proof}
\begin{proof}[Proof of Theorem \ref{type 4 stabiliser}]
First note that
\[T_2=\left\{\begin{pmatrix}1&w_1\\0&1\end{pmatrix}\middle|w_1\in\mathfrak{o}_1\right\},\]
has representation zeta polynomial $q\mathcal{D}$, as claimed. From now on, let $n\geq 3$.
Label the characters of $N_n$ by
\[\chi_{i_1,\dots,i_{n-1}}\begin{pmatrix}1&\mathbf{w}^T\\0&I\end{pmatrix}=\sigma_{i_1}(w_1)\sigma_{i_2}(w_2)\dots\sigma_{i_{n-1}}(w_{n-1}).\]
Let $Z=\begin{pmatrix}1&0\\\mathbf{v}&Y\end{pmatrix}$, and $h=\begin{pmatrix}1&0\\0&Z\end{pmatrix}$. The following formula for the conjugate holds:
\[\begin{pmatrix}1&0\\0&Z\end{pmatrix}^{-1}\begin{pmatrix}1&\mathbf{w}^T\\0&I\end{pmatrix}\begin{pmatrix}1&0\\0&Z\end{pmatrix}=\begin{pmatrix}I&\mathbf{w}^TZ\\0&1\end{pmatrix}.\]
Therefore,
\begin{align*}
^h\chi_{i_1,\dots,i_{n-1}}\begin{pmatrix}1&\mathbf{w}^T\\0&I\end{pmatrix}&=\chi_{i_1,\dots,i_{n-1}}\begin{pmatrix}I&\mathbf{w}^TZ\\0&1\end{pmatrix}\\&=\zeta_p^{\mathrm{tr}(\eta(\mathbf{w}))},
\end{align*}
where
\begin{align*}
\eta(\mathbf{w})&=i_1(w_1+v_1w_2+\dots+v_{n-2}w_{n-1})\\
&\quad +i_2(y_{11}w_2+\dots+y_{n-2,1}w_{n-1})\\
&\quad+\dots\\\
&\quad+i_{n-1}(y_{1,n-2}w_2+\dots+y_{n-2,n-2}w_{n-1})\\
&=i_1w_1+(i_1v_1+i_2y_{11}+\dots+i_{n-1}y_{1,n-2})w_2\\
&\quad+\dots\\
&\quad+(i_1v_{n-2}+i_2y_{n-2,1}+\dots+i_{n-1}y_{n-2,n-2})w_{n-1}.
\end{align*}
Therefore, there is a left action of $H_n$ on $\mathrm{Irr}(N_n)$ given explicitly by the index transformation
\[^h\chi_{i_1,\dots,i_{n-1}}=\chi_{i_1,(i_1v_1+i_2y_{11}+\dots+i_{n-1}y_{1,n-2}),\dots,(i_1v_{n-2}+i_2y_{n-2,1}+\dots+i_{n-1}y_{n-2,n-2})}.\]
For each $i_1\neq 0$, choosing $Y$ and then taking a free choice of $v_1,\dots,v_{n-2}$ gives an orbit $\chi_{i_1,*,\dots,*}=\{\chi_{i_1,\dots,i_{n-1}}\mid i_2,\dots,i_{n-1}\in\mathfrak{o}_1\}$, which we shall refer to as type (a). There are $q-1$ such orbits, one for each $i_1$.\\
If $i_1=0$, the index transformation is given by the matrix $Y^T$, which is invertible of free choice, giving an orbit \[\{\chi_{0,i_2,\dots,i_{n-1}}\mid i_k\neq 0\text{ for some }2\leq k\leq n-1\},\]
which we call type (b). The remaining orbit is the trivial one, $\{\chi_{0,\dots,0}\}$, which we call type (c).
\par We deduce equations for the stabilisers:
\begin{align*}
^h\chi_{i_1,\dots,i_{n-1}}=\chi_{i_1,\dots,i_{n-1}}&\iff\text{for all } \mathbf{w}\in\mathfrak{o}_1^{n-1},\zeta_p^{\mathrm{tr}(\eta(\mathbf{w}))}=\zeta_p^{\mathrm{tr}(i_1w_1+\dots+i_{n-1}w_{n-1})}\\
&\iff\text{for all } \mathbf{w}\in\mathfrak{o}_1^{n-1},\zeta_p^{\mathrm{tr}(\eta'(\mathbf{w}))}=1,
\end{align*}
where
\begin{align*}
\eta'(\mathbf{w})&=\eta(\mathbf{w})-(i_1w_1+\dots+i_{n-1}w_{n-1})\\
&=(i_1v_1+i_2(y_{11}-1)+\dots+i_{n-1}y_{1,n-2})w_2\\
&\quad+\dots\\
&\quad+(i_1v_{n-2}+i_2y_{n-2,1}+\dots+i_{n-1}(y_{n-2,n-2}-1))w_{n-1}.
\end{align*}
Setting all but one of $w_2,\dots,w_{n-1}$ to zero in turn, we obtain the following system of equations:
\begin{align*}
i_1v_1+i_2(y_{11}-1)+\dots+i_{n-1}y_{1,n-2}&=0\\
&\vdots\\
i_1v_{n-2}+i_2y_{n-2,1}\dots+i_{n-1}(y_{n-2,n-2}-1)&=0
\end{align*}
We find explicit forms for the stabilisers of irreducible representations of $N_n$ inside~$H_n$.\\
\underline{Type (a) ($q-1$ orbits)}\\
For $i_1\neq 0,i_2=\dots=i_{n-1}=0$, the system becomes $i_1v_1=\dots=i_1v_{n-2}=0$, therefore $v_1=\dots=v_{n-2}=0$, and
\begin{align*}
\mathrm{Stab}_H(\chi_{i_1,0,\dots,0})&=\left\{\begin{pmatrix}1&&\\&1&\\&&Y\end{pmatrix}\middle|Y\in\GL_{n-2}(\mathfrak{o}_1)\right\},
\end{align*}
with index $q^{n-2}$. For each $i_1\neq 0$,
\[\mathcal{R}_{T_n\mid\chi_{i_1,0,\dots,0}}(\mathcal{D})=\mathcal{R}_{\GL_{n-2}(\mathfrak{o}_1)}(\mathcal{D}^{q^{n-2}}),\]
therefore
\[\sum_{i_1\neq 0}\mathcal{R}_{T_n\mid\chi_{i_1,0,\dots,0}}(\mathcal{D})=(q-1)\mathcal{R}_{\GL_{n-2}(\mathfrak{o}_1)}(\mathcal{D}^{q^{n-2}}).\]
\underline{Type (b)}\\
For the non-trivial orbit with $i_1=0$, we choose $i_2=1$ and all other $i_k=0$. The system becomes $y_{11}-1=y_{21}=\dots=y_{n-2,1}=0$, therefore the stabiliser is
\begin{align*}
\mathrm{Stab}_{H_n}(\chi_{0,1,0,\dots,0})=\left\{\begin{pmatrix}1&&&&&\\&1&&&&\\&v_1&1&y_{12}&\dots&y_{1,n-2}\\&v_2&0&y_{22}&\dots&y_{2,n-2}\\&\vdots&\vdots&\vdots&\ddots&\vdots\\&v_{n-2}&0&y_{n-2,2}&\dots&y_{n-2,n-2}\end{pmatrix}\middle|\begin{aligned}v_i,y_{jk}\in\mathfrak{o}_1,\\\det\neq 0\end{aligned}\right\},
\end{align*}
with index $q^{n-2}-1$.
There is an isomorphism $\mathrm{Stab}_{H_n}(\chi_{0,1,0,\dots,0})\cong T_{n-1}$ which can be seen by conjugating by the permutation matrix corresponding to the transposition $(2,3)$. By induction, we can assume that the representations of this are known, and deal with the base case of $n=3$:
\[\Stab_{H_3}(\chi_{01})\cong T_2=\left\{\begin{pmatrix}1&w_1\\&1\end{pmatrix}\middle| w_1\in\mathfrak{o}_1\right\},\]
which has representation zeta polynomial
\[\mathcal{R}_{T_2}(\mathcal{D})=q\mathcal{D}.\]
For general $n$, we have
\[\mathcal{R}_{T_n\mid\chi_{0,1,0,\dots,0}}(\mathcal{D})=\mathcal{R}_{T_{n-1}}(\mathcal{D}^{q^{n-2}-1}).\]
\underline{Type (c)}\\
The group $H_n$ is isomorphic to $P_{n-1}$. Consider the case of the trivial representation,
\[\mathrm{Stab}_{H_n}(\chi_{0,\dots,0})=H_n.\]
The contribution to the representation zeta polynomial of $T_n$ is
\begin{align*}
\mathcal{R}_{T_n\mid\chi_{0,\dots,0}}(\mathcal{D})&=\mathcal{R}_{P_n}(\mathcal{D}),
\end{align*}
which we know by Proposition \ref{representations of P_n}.
\par Collecting the representation zeta polynomials above the representatives and adding them (see Remark \ref{sum of zeta polynomials above representatives}) will give the representation zeta polynomial of $T_n$ as claimed.
\end{proof}
\section{The Representation Zeta Polynomial of $G_{(\ell,1^n)}$}\label{arbitrary characteristic}
Using the method of proof of \cite[Lemma 7.7]{singla10} (an argument attributed to U.~Onn), we derive the representation zeta polynomial of $G_{(\ell,1^n),\mathfrak{o}}$, thus proving Theorem \ref{main theorem}. 
\begin{proposition}\label{heisenberg lift}\cite{bushnellfrohlich83}
Let $G$ be a finite group and $N$ a normal subgroup of $G$ such that $V=G/N$ is an elementary abelian $p$-group, regarded as an $\mathbb{F}_p$-vector space. Let $\chi\in\Irr(N)$ be invariant in $G$ and suppose that the bilinear form \[h_\chi:V\times V\to\mu_p(\mathbb{C})\cong\mathbb{F}_p;\quad h_\chi( g_1N,g_2N)=\chi([g_1,g_2])\]
is non-degenerate. Then there exists a unique $\rho_\chi\in\Irr(G)$ such that $\langle\chi,\rho_\chi|_N\rangle\neq 0$. Moreover, $\rho_\chi|_N=e\chi$ for some $e\geq 1$, and $\dim \rho_\chi=[G:N]^{\frac 12}$.
\end{proposition}
\begin{proof}
Choose a maximal isotropic subspace $J/N$ for the bilinear form. The representation $\rho_\chi$ is obtained by extending $\chi$ to $J$ and then inducing to $G$. One can show that $\rho_\chi$ does not depend on the choice of $J$ or the extension of~$\chi$.
\end{proof}
\begin{lemma}\label{extending linear characters}\cite[Lemma 5.4]{singla10}
Let $G$ be a finite group, $N\lhd G$ and $M\leq G$ such that~$G=NM$. Let $\psi_1,\psi_2$ be one-dimensional representations of $N$ and $M$ respectively such that $\psi_1$ is invariant in $M$ and $\psi_1|_{M\cap N}=\psi_2|_{M\cap N}$. Then $\psi_1\psi_2\in\Irr(G)$ defined by $\psi_1\psi_2(nm)=\psi_1(n)\psi_2(m)$ is the unique one-dimensional representation of $G$ extending both $\psi_1$ and $\psi_2$.
\end{lemma}
\begin{proof}[Proof of Theorem \ref{main theorem}]
We can use Proposition \ref{heisenberg lift} to construct representations of~$G_{(\ell,1^n)}$, following the method in \cite[Lemma 7.7]{singla10}. Denote by $\mathfrak{p}_\ell^{\ell-1}$ the image of $\mathfrak{p}$ in $\mathfrak{o}_\ell$. Fix a non-trivial character $\psi:\mathfrak{o}_1\to\mathbb{C}^\times$ and define the following groups:
\[G=\begin{pmatrix}\mathfrak{o}_\ell^\times&\mathfrak{p}_\ell^{\ell-1}&\dots&\mathfrak{p}_\ell^{\ell-1}\\\mathfrak{o}_1&&&\\\vdots&&\GL_n(\mathfrak{o}_1)&\\\mathfrak{o}_1&&&\end{pmatrix},H=\begin{pmatrix}1+\mathfrak{p}_\ell^{\ell-1}&\mathfrak{p}_\ell^{\ell-1}&\dots&\mathfrak{p}_\ell^{\ell-1}\\\mathfrak{o}_1&1&&\\\vdots&&\ddots&\\\mathfrak{o}_1&&&1\end{pmatrix}.\]
Then $H$ is a normal subgroup of $G$
with centre
\[Z(H)=\begin{pmatrix}1+\mathfrak{p}_\ell^{\ell-1}&&&\\&1&&\\&&\ddots&\\&&&1\end{pmatrix}\cong\mathfrak{o}_1,\]
such that $H/Z(H)$ is an elementary abelian $p$-group. The non-trivial characters of $Z(H)$ are parameterised for $z\in Z(H),z\neq 1$ by
\[\psi_z\begin{pmatrix}1+\varpi^{\ell-1} v&&&\\&1&&\\&&\ddots&\\&&&1\end{pmatrix}=\psi(xv),\]
where
\[z=\begin{pmatrix}1+\varpi^{\ell-1} x&&&\\&1&&\\&&\ddots&\\&&&1\end{pmatrix}.\]
By explicit computation of commutators in $H$, one can check that the alternating bilinear form on $H/Z(H)$ given by
\[\langle\overline{h_1},\overline{h_2}\rangle_{\psi_z}=\psi_z([h_1,h_2])\]
is non-degenerate. Applying Proposition \ref{heisenberg lift}, there are $q-1$ pairwise inequivalent irreducible representations of $H$ lying over the $q-1$ distinct non-trivial characters of $Z(H)$, namely $\rho_\chi$ for each non-trivial character $\chi:Z(H)\to\mathbb{C}^\times$. The centre of $G$ is
\[Z(G)=\left\{\begin{pmatrix}u&\\&\overline{u}I\end{pmatrix}\middle| u\in\mathfrak{o}_\ell^\times\right\}.\]
In particular, $Z(H)\leq Z(G)$ and $Z(H)$ is normal in $G$. For $g\in G,h\in Z(H)$,
\[^g\rho_\chi(h)=\rho_\chi(g^{-1}hg)=\rho_\chi(h),\]
therefore $\Res^H_{Z(H)} {^g\rho_\chi}=\Res^H_{Z(H)}\rho_\chi=\chi$. Thus, $^g\rho_\chi=\rho_\chi$ and $\rho_\chi$ is invariant in $G$.
\par In the following discussion, fix a non-trivial character $\chi\in\Irr(Z(H))$. The subgroup
\[\overline{H^\mathrm{iso}}=\begin{pmatrix}1&&&\\\mathfrak{o}_1&1&&\\\vdots&&\ddots&\\\mathfrak{o}_1&&&1\end{pmatrix}\leq G/Z(H)\]
is a maximal isotropic subspace for $\langle,\rangle_{\psi_z}$. Consider the inverse image
\[H^\mathrm{iso}=\begin{pmatrix}1+\mathfrak{p}_\ell^{\ell-1}&&&\\\mathfrak{o}_1&1&&\\\vdots&&\ddots&\\\mathfrak{o}_1&&&1\end{pmatrix}\]
and define the character $\chi^\mathrm{iso}\in\Irr(H^\mathrm{iso})$ by \[\chi^\mathrm{iso}\begin{pmatrix}u&\\w&1\end{pmatrix}=\chi(u).\] Choose an extension $\tilde{\chi}$ of $\chi$ to $\mathfrak{o}_\ell^\times\times\GL_n(\mathfrak{o}_1)=:M$.
Define $G^\mathrm{iso}=H^\mathrm{iso}M$ so that
\[G^\mathrm{iso}=\begin{pmatrix}\mathfrak{o}_\ell^\times&0&\dots&0\\\mathfrak{o}_1&&&\\
\vdots&&\GL_n(\mathfrak{o}_1)&\\\mathfrak{o}_1&&&\end{pmatrix}.\]
Observe that $H^\mathrm{iso}$ is normal in $G^\mathrm{iso}$. Further, we show that $\chi^\mathrm{iso}$ is invariant in $M$, and that $\chi^\mathrm{iso}|_{H^\mathrm{iso}\cap M}=\tilde\chi|_{H^\mathrm{iso}\cap M}$ (for then we can apply Lemma \ref{extending linear characters}).
\par Let $m=\begin{pmatrix}y&\\&D\end{pmatrix}\in M$. Then
\begin{align*}
^m\chi^\mathrm{iso}\begin{pmatrix}u&\\w&1\end{pmatrix}&=\chi^\mathrm{iso}\left(\begin{pmatrix}y&\\&D\end{pmatrix}\begin{pmatrix}u&\\w&1\end{pmatrix}\begin{pmatrix}y^{-1}&\\&D^{-1}\end{pmatrix}\right)\\
&=\chi^\mathrm{iso}\begin{pmatrix}u&\\\overline{y}^{-1}Dw&1\end{pmatrix}\\
&=\chi(u).
\end{align*}
Therefore, $\chi^\mathrm{iso}$ is invariant in $M$. Finally, $H^\mathrm{iso}\cap M=Z(H)$, therefore \[\chi^\mathrm{iso}|_{H^\mathrm{iso}\cap M}=\chi=\tilde{\chi}|_{H^\mathrm{iso}\cap M}.\] Thus, by Lemma \ref{extending linear characters}, we get a unique linear character $\chi^\mathrm{iso}\tilde{\chi}:G^\mathrm{iso}\to\mathbb{C}^\times$.
\par Note that $\dim\rho_\chi=[H:Z(H)]^{\frac 12}=q^n$. Define $\rho_{\chi^\mathrm{iso}}=\Ind^G_{G^\mathrm{iso}}\chi^\mathrm{iso}\tilde{\chi}$ and note also that $\dim\rho_{\chi^\mathrm{iso}}=q^n$. We show that $\rho_{\chi^\mathrm{iso}}$ contains $\rho_\chi$. Indeed, since $\rho_{\chi^\mathrm{iso}}$ contains~$\chi$, $\rho_{\chi^\mathrm{iso}}|_H$ also contains $\chi$.
Since $\rho_\chi$ is the unique irreducible representation of $H$ containing $\chi$, $\rho_{\chi^\mathrm{iso}}|_H$ contains $\rho_\chi$. Further, $\rho_{\chi^\mathrm{iso}}$ and $\rho_\chi$ have equal dimension, therefore $\rho_{\chi^\mathrm{iso}}|_H=\rho_\chi$, that is, $\rho_\chi$ extends to $\rho_{\chi^\mathrm{iso}}\in\Irr(G)$.
\par By Proposition \ref{Clifford theory},
\[\Irr(G\mid\rho_\chi)=\{\rho_{\chi^\mathrm{iso}}\phi\mid\phi\in\Irr(G/H)\},\]
where $G/H\cong \mathfrak{o}_{\ell-1}^\times\times\GL_n(\mathfrak{o}_1)$. Therefore,
\begin{align*}\sum_{\substack{\chi\in\Irr(Z(H))\\\chi\neq 1}}\mathcal{R}_{G\mid\rho_\chi}(\mathcal{D})&=(q-1)\mathcal{R}_{\mathfrak{o}_{\ell-1}^\times\times\GL_n(\mathfrak{o}_1)}(\mathcal{D}^{q^n})\\
&=q^{\ell-2}(q-1)^2\mathcal{R}_{\GL_n(\mathfrak{o}_1)}(\mathcal{D}^{q^n}).
\end{align*}
This gives an expression for the representation zeta polynomial of $G_{(\ell,1^n)}$:
\begin{align*}\mathcal{R}_{G_{(\ell,1^n)}}(\mathcal{D})&=q^{\ell-2}(q-1)^2\mathcal{R}_{\GL_n(\mathfrak{o}_1)}(\mathcal{D}^{q^n})+\mathcal{R}_{G/Z(H)}(\mathcal{D}).\end{align*}
The quotient $G/Z(H)$ splits as the direct product
\[G/Z(H)=\left\{\begin{pmatrix}u&\\&\overline{u}\end{pmatrix}\middle|u\in\mathfrak{o}_{\ell-1}^\times\right\}\times\begin{pmatrix}1&\mathfrak{p}_\ell^{\ell-1}&\dots&\mathfrak{p}_\ell^{\ell-1}\\\mathfrak{o}_1&&&\\\vdots&&\GL_n(\mathfrak{o}_1)&\\\mathfrak{o}_1&&&\end{pmatrix}\]
which we can write as $\mathfrak{o}_{\ell-1}^\times\times((\mathfrak{o}_1^n\times\mathfrak{o}_1^n)\rtimes G_{(1^n)})$, therefore
\begin{align*}\mathcal{R}_{G_{(\ell,1^n)}}(\mathcal{D})&=q^{\ell-2}(q-1)^2\mathcal{R}_{\GL_n(\mathfrak{o}_1)}(\mathcal{D}^{q^n})+q^{\ell-2}(q-1)\mathcal{R}_{((\mathfrak{o}_1^n\times\mathfrak{o}_1^n)\rtimes G_{(1^n)})}(\mathcal{D}).\end{align*}
The final task is to find $\mathcal{R}_{((\mathfrak{o}_1^n\times\mathfrak{o}_1^n)\rtimes G_{(1^n)})}(\mathcal{D})$. The action in the semidirect product is the conjugation action
\[\begin{pmatrix}1&\\&\GL_n(\mathfrak{o}_1)\end{pmatrix}\curvearrowright\begin{pmatrix}1&\mathfrak{p}_\ell^{\ell-1}&\dots&\mathfrak{p}_\ell^{\ell-1}\\\mathfrak{o}_1&1&&&\\\vdots&&\ddots&\\\mathfrak{o}_1&&&1\end{pmatrix}\cong \mathfrak{o}_1^n\times\mathfrak{o}_1^n\]
given by
\[\begin{pmatrix}1&\\&D\end{pmatrix}\begin{pmatrix}1&\varpi^{\ell-1} v^T\\w&1\end{pmatrix}\begin{pmatrix}1&\\&D^{-1}\end{pmatrix}=\begin{pmatrix}1&\varpi^{\ell-1}v^TD^{-1}\\Dw&1\end{pmatrix}.\]
As in \cite{singla10}, this can be identified with the action $\GL_n(\mathfrak{o}_1)\curvearrowright(\mathfrak{o}_1^n\times\mathfrak{o}_1^n)$ given by
\[g^{-1}(\hat{v},\hat{w})=(D^{-1}\hat{v},\hat{w}D),\text{ where }g=\begin{pmatrix}1&\\&D\end{pmatrix}.\]
The orbits and stabilisers are given in Table \ref{action}, from which the representation zeta polynomial of $(\mathfrak{o}_1^n\times\mathfrak{o}_1^n)\rtimes G_{(1^n)}$ can now be deduced as in the statement of Theorem~\ref{main theorem}.
\begin{table}
\centering
\begin{tabular}{ccc}
Orbit representative $\chi$&$\Stab_{G_{(1^n)}}(\chi)$&$[G_{(1^n)}:\Stab_{G_{(1^n)}}(\chi)]$\\\hline
$\left[\begin{pmatrix}0\\\vdots\\0\end{pmatrix},\begin{pmatrix}0&\dots&0\end{pmatrix}\right]$&$G_{(1^n)}$&1\\
$\left[\begin{pmatrix}0\\\vdots\\0\end{pmatrix},\begin{pmatrix}1&0&\dots&0\end{pmatrix}\right]$&$\begin{pmatrix}1&0\\\mathfrak{o}_1^{n-1}&G_{(1^{n-1})}\end{pmatrix}\cong P_n$&$q^n-1$\\
$\left[\begin{pmatrix}1\\0\\\vdots\\0\end{pmatrix},\begin{pmatrix}0&\dots&0\end{pmatrix}\right]$&$\begin{pmatrix}1&\mathfrak{o}_1^{n-1}\\0&G_{(1^{n-1})}\end{pmatrix}\cong P_n$&$q^n-1$\\
$\left[\begin{pmatrix}1\\0\\\vdots\\0\end{pmatrix},\begin{pmatrix}0&1&0&\dots&0\end{pmatrix}\right]$&$T_n$&$(q^n-1)(q^{n-1}-1)$\\
$\left[\begin{pmatrix}1\\0\\\vdots\\0\end{pmatrix},\begin{pmatrix}u&0&\dots&0\end{pmatrix}\right];u\in\mathfrak{o}_1^\times$&$\begin{pmatrix}1&\\&G_{(1^{n-1})}\end{pmatrix}$&$q^{n-1}(q^n-1)$
\end{tabular}
\caption{The action of $G_{(1^n)}$ on $\mathfrak{o}_1^n\times\mathfrak{o}_1^n$}
\label{action}
\end{table}
\end{proof}
\begin{example}
One can check by substituting $\ell=2,n=2$ that the expression obtained for $\mathcal{R}_{G_{(2,1,1)}}(\mathcal{D})$ agrees with \cite[Lemma 7.7]{singla10}. Moreover, for $\ell=3,n=2$ one can calculate
\begin{align*}\mathcal{R}_{G_{(3,1,1)}}(\mathcal{D})&=q(q-1)^2\mathcal{D}+\frac 12 q^2(q-1)^2\mathcal{D}^{q-1}+q(q-1)^2\mathcal{D}^q\\&\quad+\frac 12 q(q-1)^2(q-2)\mathcal{D}^{q+1}+2q(q-1)^2\mathcal{D}^{(q-1)(q+1)}\\&\quad+q(q-1)^3\mathcal{D}^{q^2}+q(q-1)(q+2)\mathcal{D}^{(q-1)^2(q+1)}\\&\quad+\frac 12 q^2(q-1)^3\mathcal{D}^{q^2(q-1)}+q(q-1)^3\mathcal{D}^{q(q-1)(q+1)}\\&\quad+q(q-1)^3\mathcal{D}^{q^3}+\frac 12 q(q-1)^3(q-2)\mathcal{D}^{q^2(q+1)}.\end{align*}
\end{example}
Further, $\mathcal{R}_{G_{(2,1,1,1)}}(\mathcal{D})$ has 24 terms, with the polynomials giving the dimensions of degree at most 6. We omit the expression due to its length.

The smallest value of $n$ for which the dimensions of the irreducible representations of $G_{(\ell,1^n),\mathfrak{o}}$ are not known to be polynomial for some $\lambda$ is $n=5$. This is settled in the affirmative for the Jordan canonical forms of type: $(5)$ (abelian), $(1,1,1,1,1)$ by Green \cite{green}, $(4,1)$ and $(3,2)$ by Onn \cite{onn08}, and $(2,1,1,1)$ and $(3,1,1)$ by Theorem \ref{main theorem}. The question is still open in the case of $(2,2,1)$.
\par Acknowledgements: The author was supported by the HIMR/UKRI Additional Funding Programme for Mathematical Sciences, EP/V521917/1. The author also gratefully acknowledges the supervision of Alexander Stasinski and Jack Shotton. The author is grateful to Pooja Singla for helpful comments.
\bibliographystyle{plain}
\bibliography{reps_ell_1_n.bib}

\begin{thebibliography}{10}

\bibitem{bushnellfrohlich83}
Colin~J. Bushnell and Albrecht Fr\"ohlich.
\newblock {\em Gauss sums and {$p$}-adic division algebras}, volume 987 of {\em
  Lecture Notes in Mathematics}.
\newblock Springer-Verlag, Berlin-New York, 1983.

\bibitem{clifford37}
A.~H. Clifford.
\newblock Representations induced in an invariant subgroup.
\newblock {\em Ann. of Math. (2)}, 38(3):533--550, 1937.

\bibitem{green}
J.~A. Green.
\newblock The characters of the finite general linear groups.
\newblock {\em Trans. Amer. Math. Soc.}, 80:402--447, 1955.

\bibitem{hadas24}
Itamar Hadas.
\newblock Spectral equivalence of smooth group schemes over principal ideal
  local rings.
\newblock {\em Journal of Algebra}, 2024.

\bibitem{hilljordan}
Gregory Hill.
\newblock A {J}ordan decomposition of representations for
  {$\mathrm{GL}_n(\mathscr{O})$}.
\newblock {\em Comm. Algebra}, 21(10):3529--3543, 1993.

\bibitem{hillnilpotent}
Gregory Hill.
\newblock On the nilpotent representations of {$\mathrm{GL}_n(\mathscr{O})$}.
\newblock {\em Manuscripta Math.}, 82(3-4):293--311, 1994.

\bibitem{hill95}
Gregory Hill.
\newblock Regular elements and regular characters of
  {$\mathrm{GL}_n(\mathscr{O})$}.
\newblock {\em J. Algebra}, 174(2):610--635, 1995.

\bibitem{hillsemisimplecuspidal}
Gregory Hill.
\newblock Semisimple and cuspidal characters of {$\mathrm{GL}_n(\mathscr{O})$}.
\newblock {\em Comm. Algebra}, 23(1):7--25, 1995.

\bibitem{huppert98}
Bertram Huppert.
\newblock {\em Character theory of finite groups}, volume~25 of {\em De Gruyter
  Expositions in Mathematics}.
\newblock Walter de Gruyter \& Co., Berlin, 1998.

\bibitem{isaacs}
Irving~Martin Isaacs.
\newblock {\em Character theory of finite groups}.
\newblock Pure and applied mathematics (Academic Press) ; 69. Academic Press,
  New York, 1976.

\bibitem{j24}
Alexander Jackson.
\newblock A polynomial result for dimensions of irreducible representations of
  smooth affine group schemes over principal ideal local rings.
\newblock arXiv:2405.13724 [math.RT], 2024.

\bibitem{onn08}
Uri Onn.
\newblock Representations of automorphism groups of finite
  {$\mathfrak{o}$}-modules of rank two.
\newblock {\em Adv. Math.}, 219(6):2058--2085, 2008.

\bibitem{serre71}
Jean-Pierre Serre.
\newblock {\em Repr\'{e}sentations lin\'{e}aires des groupes finis}.
\newblock Hermann, Paris, 1971.
\newblock Deuxi\`eme \'{e}dition, refondue.

\bibitem{singla10}
Pooja Singla.
\newblock On representations of general linear groups over principal ideal
  local rings of length two.
\newblock {\em J. Algebra}, 324(9):2543--2563, 2010.

\bibitem{stasinskisurvey}
Alexander Stasinski.
\newblock Representations of {${\rm GL}_N$} over finite local principal ideal
  rings: an overview.
\newblock In {\em Around {L}anglands correspondences}, volume 691 of {\em
  Contemp. Math.}, pages 337--358. Amer. Math. Soc., Providence, RI, 2017.

\end{thebibliography}
\end{document}